\documentclass[twocolumn]{article}

\usepackage{amsmath,amsfonts,amsthm,amscd,amssymb}  
\usepackage{dsfont}			
\usepackage{thmtools}		
\usepackage{enumerate}

\usepackage[utf8]{inputenc} 

\usepackage{graphicx}			
\usepackage{subcaption}			
\usepackage[font=small]{caption}	

\usepackage[hyperindex,breaklinks]{hyperref}	
\hypersetup{colorlinks=true, linkcolor=blue, citecolor=blue}

\usepackage[
disable,					
textsize=scriptsize,textwidth=4.2cm]{todonotes}	
\newcommand{\MAYBE}[1]{\todo[color=orange!60]{#1}}  
\newcommand{\SELF}[1]{\todo[color=green!40]{#1}} 
\newcommand{\CITE}[1]{\todo[color=cyan!30]{#1}}  

\newtheorem*{theorem*}{Teorema}
\newtheorem*{corollary*}{Corolário}

\theoremstyle{definition}
\newtheorem*{example}{Examplo} 

\newcommand\inner[1] 		{\langle #1 \rangle}

\newcommand{\R}{\mathds{R}}
\newcommand{\C}{\mathds{C}}

\newcommand{\II}{\mathcal{I}}
\newcommand{\AAA}{\mathcal{A}}
\newcommand{\VV}{\mathcal{V}}
\newcommand{\LL}{\mathcal{L}}
\newcommand{\im}{\mathrm{i}}				%

\DeclareMathOperator{\Span}{span}
\DeclareMathOperator{\sen}{sen}
\DeclareMathOperator{\Div}{div}

\begin{document}

\title{As Mil Faces de Pitágoras}

\author{Andr\'e L. G. Mandolesi}


\maketitle



O Teorema de Pitágoras (TP) é um dos mais antigos, famosos e úteis teoremas da Matemática, e possivelmente o que maior impacto teve na evolução desta e outras ciências.
Neste artigo, vamos olhar para este velho conhecido de diferentes perspectivas, algumas pouco usuais.
Iremos lembrar um pouco da sua história, algumas aplicações e generalizações bem conhecidas, outras nem tanto,
e ver que ele guarda muitas facetas surpreendentes e geralmente ignoradas.

\section{O teorema de mil faces}

Em sua forma moderna, o TP é a relação 
\begin{equation}\label{eq:TP}
	a^2 = b^2 + c^2,
\end{equation}
entre os comprimentos $a$ da hipotenusa, e $b$ e $c$ dos catetos, de um triângulo retângulo (Fig.\,\ref{fig:TP}).

\begin{figure}[htb!]
	\centering
	\includegraphics[width=3cm]{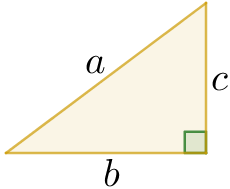}
	\caption{Teorema de Pitágoras, $a^2=b^2+c^2$}
	\label{fig:TP}
\end{figure}

Sua importância se reflete nas centenas de demonstrações surgidas ao longo do tempo \cite{Bogomolny,Loomis1968}.
E a forma como ele permeia toda a Matemática, sob diferentes roupagens, se reflete na diversidade de métodos que usaremos: geometria euclidiana, não-euclidiana, analítica, Riemanniana, trigonometria circular e hiperbólica, álgebra linear e exterior, determinantes, cálculo integral e vetorial, combinatória, números complexos, etc.
Isso mostra também a unidade da Matemática, e como é fértil o fluxo de conhecimentos entre suas diferentes áreas.

Inúmeros resultados bem conhecidos são basicamente aplicações ou generalizações do TP.
Sendo impossível listar todos, ficaremos só nos mais óbvios:

\begin{itemize}
	\item O TP equivale à identidade trigonométrica fundamental, $\cos^2 \theta + \sen^2 \theta = 1$.
%
%

	\item Decompondo $v\in \R^n$ ou $\C^n$ numa base ortogonal como $v = v_1+\cdots +v_n$, temos
	\begin{equation}\label{eq:norma}
		\|v \|^2 = \|v_1\|^2+\cdots +\|v_n\|^2.
	\end{equation}
	
	\item Em $\R^n$ ou $\C^n$, a distância $d$ entre $A=(a_1,\ldots,a_n)$ e $B=(b_1,\ldots,b_n)$  é
	\begin{equation*}\label{eq:distancia}
		d(A,B)^2=\sum\limits_{k=1}^n |a_k-b_k|^2.
	\end{equation*}
	
	\item Se os lados $b$ e $c$ formarem um ângulo $\theta$, o TP se generaliza como a Lei dos Cossenos,
	\begin{equation}\label{eq:lei cos}
		a^2 = b^2 + c^2 - 2bc\cos \theta.
	\end{equation}

	\item A métrica Riemanniana no plano euclidiano é $ds^2=d x^2+dy^2$, um TP para distâncias infinitesimais.
	Em superfícies curvas (Fig.\,\ref{fig:surface}) ela é uma versão infinitesimal de \eqref{eq:lei cos},
	\[ ds^2 = g_{xx} d x^2 + g_{yy} dy^2 + 2g_{xy} dx dy, \]
	onde $g_{xx}$ e $g_{yy}$ ajustam a escala de comprimento nas direções $x$ e $y$, e $g_{xy}$ envolve o ângulo entre elas.
	Em dimensão $n$, $ds^2=\sum_{i,j=1}^n g_{ij} dx_i dx_j$.
	
	\begin{figure}[htb!]
		\centering
		\includegraphics[width=6cm]{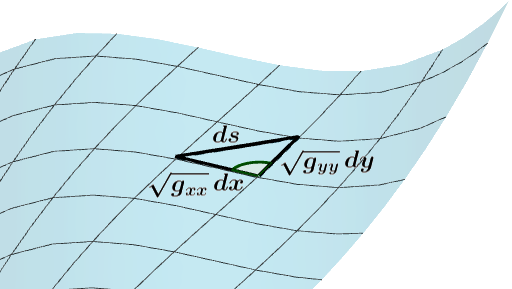}
		\caption{Métrica Riemanniana em uma superfície é uma versão infinitesimal da Lei dos Cossenos}
		\label{fig:surface}
	\end{figure}
	
	\item Num espaço vetorial com produto interno $\inner{\cdot,\cdot}$, 
	$u$ e $v$ são ditos ortogonais se $\inner{u,v}=0$, e nesse caso $\|u+v\|^2=\|u\|^2+\|v\|^2$.
	Com isso pode-se até, por exemplo, interpretar $\int_0^{2\pi} (\sen x + \cos x)^2 \,dx = \int_0^{2\pi} \sen^2 x \,dx + \int_0^{2\pi} \cos^2 x \,dx$ em termos do TP num espaço de funções.
	

	
\end{itemize}

Veremos depois generalizações menos conhecidas.%

\section{Um pouco de história}

Apesar do nome, o TP era conhecido por várias civilizações muito antes da época de Pitágoras (séc.\ VI a.C.),
sendo improvável que ele ou seus seguidores, os pitagóricos, tenham sido os primeiros a prová-lo \cite{Roque2012}.
E é possível que estivessem menos interessados em \eqref{eq:TP} como resultado geométrico do que como uma relação aritmética de certos trios de números naturais,
as \emph{ternas pitagóricas}, como $(3,4,5)$, $(5,12,13)$, etc.

Estes trios eram tabelados por vários povos antigos,
e possivelmente tinham finalidade didática, servindo como exemplos fáceis de trabalhar.
Também podiam ter utilidade prática: para demarcar um terreno retangular, uma maneira de checar se os lados estão bem retos e perpendiculares é ver se formam com a diagonal uma dessas ternas.
Arqueólogos ainda fazem algo parecido hoje em dia.

Os pitagóricos se interessavam menos por geometria que pelos números,
vistos como algo místico, a base do Universo, cujas propriedades revelariam verdades fundamentais.
Mas, na época, os únicos números eram os naturais, e mesmo frações eram só relações entre eles: se um segmento $a$ media $1$, dizer que outro segmento $b$ media $\frac{2}{3}$ significava apenas que 3 segmentos $b$ mediam o mesmo que 2 $a$'s.
Ou, equivalentemente, que era possível achar uma unidade comum de comprimento que coubesse 3 vezes em $a$, e 2 vezes em $b$.
E, como tudo no mundo devia ser descrito por relações entre números, esperava-se que quaisquer comprimentos pudessem ser comparados assim.

Ironicamente, o TP pode ter levado ao fim da filosofia pitagórica, e à primeira crise da história da matemática.
Usando \eqref{eq:TP} para achar a razão $\frac mn$ entre a diagonal de um quadrado e seu lado, obtém-se $\frac{m^2}{n^2} = 2$.
Mas fatorando $m$ e $n$ em primos, ao elevá-los ao quadrado a quantidade de cada fator dobra, e simplificando não há como sobrar um único $2$.
Ou seja, é impossível comparar tais comprimentos por meio de uma razão entre números (naturais).
Eles são \emph{incomensuráveis}, não podem ser ambos múltiplos inteiros de uma unidade de medida comum.

Diz-se que encontrar algo que não podia ser descrito por números foi um choque para os pitagóricos, que tentaram esconder o resultado.
Mas o se\-gre\-do se espalhou, destruindo a ideia de um mundo regido por números.
Estes deixaram o palco principal da matemática grega, que se voltou para a geometria, vista a partir daí como a base de verdades mais fundamentais, que os números não alcançam.

Alguns historiadores discordam dessa narrativa,
pois como os pitagóricos se interessavam pouco por geometria, não teriam dado tanta importância aos incomensuráveis \cite{Roque2012}.
	\CITE{p.125}
Mas é inegável o impacto que a descoberta destes teve na matemática  grega.
A percepção de que a geometria guardava mistérios contraintuitivos levou a uma maior preocupação com demonstrações formais, como na obra de Euclides.
E ela passou a ser feita quase sem números: achar uma área não era atribuir-lhe um valor, mas construir, com régua não numerada e compasso, um quadrado de mesma área, que pudesse ser comparado com outros.
Daí vem o problema da quadratura do círculo, que perseguiu os matemáticos por milênios.

\section{Áreas de figuras semelhantes}

Na geometria grega, o TP era mais comumente visto como uma relação 
\begin{equation}\label{eq:soma areas}
	\AAA = \AAA_1 + \AAA_2,
\end{equation}
entre a área $\AAA$ de um quadrado construído sobre a hipotenusa do triângulo retângulo, e áreas $\AAA_1$ e $\AAA_2$ de quadrados construídos sobre os catetos (Fig.\,\ref{fig:quadriculado}).
Nessa perspectiva, o triângulo vira um ``somador'' de áreas, juntando $\AAA_1$ e $\AAA_2$ num único quadrado  $\AAA$ com a área total.
E \eqref{eq:soma areas} só se converte em \eqref{eq:TP} após aplicarmos a fórmula da área do quadrado.

\begin{figure}[htb!]
	\centering
	\includegraphics[width=4cm]{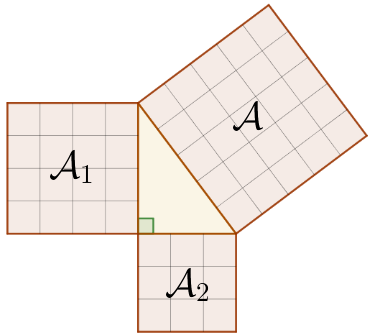}
	\caption{Teorema de Pitágoras, $\AAA = \AAA_1 + \AAA_2$.}
	\label{fig:quadriculado}
\end{figure}

Diz a lenda que Pitágoras teria descoberto \eqref{eq:soma areas} observando uma parede ladrilhada (Fig.\,\ref{fig:ladrilho}).
Uma prova bem conhecida e autoexplicativa é a da Fig.\,\ref{fig:cortes}.

\begin{figure}[htb!]
	\centering
	\includegraphics[width=4cm]{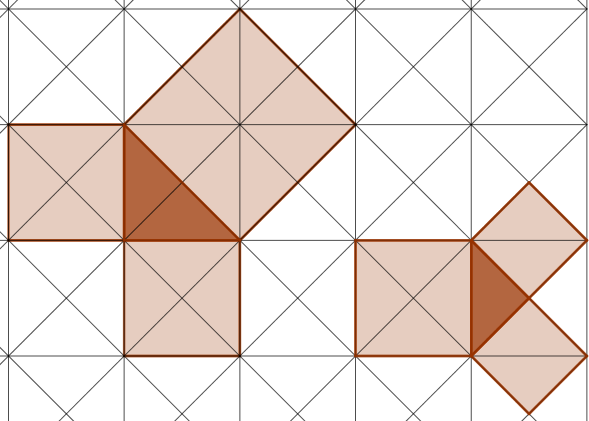}
	\caption{Teorema de Pitágoras em ladrilhos triangulares}
	\label{fig:ladrilho}
\end{figure}

\begin{figure}[htb!]
	\centering
	\includegraphics[width=6cm]{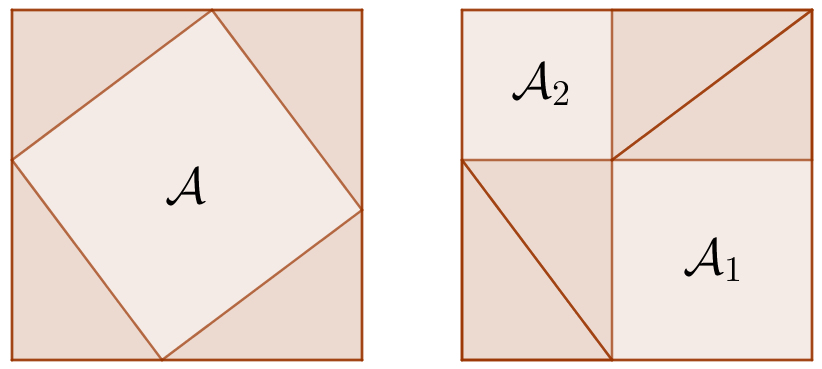}
	\caption{Duas maneiras de remover 4 triângulos retângulos iguais de um mesmo quadrado, sobrando áreas iguais}
	\label{fig:cortes}
\end{figure}

As figuras construídas sobre os lados nem precisam ser quadrados, basta serem semelhantes entre si, pois nesse caso as razões entre suas áreas e as dos quadrados serão iguais (Fig.\, \ref{fig:areageral}).
Isso também significa que o TP pode ser provado usando qualquer trio de figuras semelhantes construídas sobre os lados.

\begin{figure}[htb!]
	\centering
	\includegraphics[width=4cm]{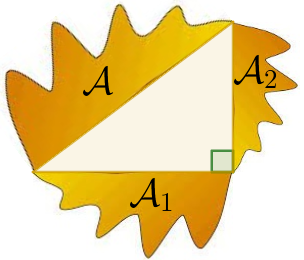}
	\caption{Áreas de figuras semelhantes, $\AAA = \AAA_1 + \AAA_2$.}
	\label{fig:areageral}
\end{figure}

Em geral, construímos as figuras para fora do triângulo,
mas as áreas não mudam se estiverem pro outro lado, certo? 
Isso levou Albert Einstein, aos 12 anos, a uma prova interessante (Fig.\,\ref{fig:Einstein}):
a altura em relação à hipotenusa divide o triângulo em dois semelhantes a ele,
e tomando os três como sendo as figuras construídas sobre os lados, é imediato que $\AAA = \AAA_1 + \AAA_2$.

\begin{figure}[htb!]
	\centering
	\includegraphics[width=3cm]{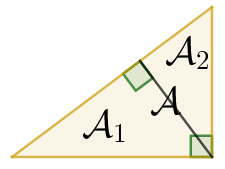}
	\caption{Demonstração de Einstein, $\AAA = \AAA_1 + \AAA_2$}
	\label{fig:Einstein}
\end{figure}

E, claro, as figuras nem precisam ser construídas nos lados do triângulo.
Basta serem figuras semelhantes, com escalas de comprimento proporcionais aos comprimentos de tais lados (Fig.\,\ref{fig:3 discos}).

\begin{figure}[htb!]
	\centering
	\includegraphics[width=7cm]{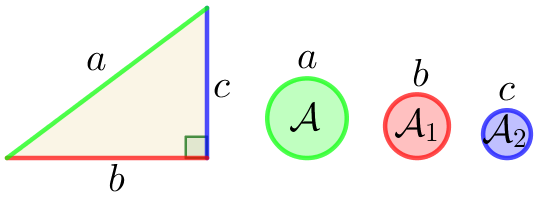}
	\caption{Discos obtidos `enrolando' os lados do triângulo retângulo. Novamente, $\AAA = \AAA_1 + \AAA_2$.}
	\label{fig:3 discos}
\end{figure}

\section{TP e o Axioma das Paralelas}

É comum provar \eqref{eq:TP} usando propriedades de semelhança e proporcionalidade, que dependem do famoso Axioma das Paralelas (AP).
Já \eqref{eq:soma areas} parece resultar, em provas como a da Fig.\,\ref{fig:cortes}, de um jogo de cortar, mover e reagrupar figuras, via congruência.
Isso pode dar a impressão de que o AP só entre na passagem de \eqref{eq:soma areas} para \eqref{eq:TP}, que usa a fórmula $\AAA=l^2$ da área do quadrado de lado $l$, obtida dividindo-o em outros menores por meio de paralelas.

Na verdade, o AP já está presente em \eqref{eq:soma areas} de várias formas. Para que as áreas $\AAA$, $\AAA_1$ e $\AAA_2$ na Fig.\,\ref{fig:cortes} tenham ângulos retos é preciso que a soma dos ângulos dos triângulos seja $180^\circ$, o que requer o AP.
E a própria existência de quadrados depende dele, bem como o conceito de figuras semelhantes: sem ele não há como dilatar uma figura de forma homogênea, preservando seus ângulos, dilatando seus comprimentos por um fator $k$, suas áreas por $k^2$, etc.

Pode-se provar que o TP, em ambas as formas, é equivalente ao AP.
E a melhor forma de entender como ele é um resultado puramente euclidiano é ver o que ocorre em outras geometrias.

\section{Geometrias não-euclidianas}

As geometrias não-euclidianas \cite{Wolfe2012} (esférica/elíptica 
e hiperbólica) resultam da negação do AP
(e alguns ajustes em outros axiomas).
Vamos começar pela esférica, que é mais intuitiva.

O que faz o papel de `retas' na superfície de uma esfera são os círculos máximos (com mesmo centro e raio da esfera), pois o menor caminho entre 2 pontos ao longo da superfície é um arco deles.
Como tais círculos sempre se interceptam, nessa geometria não há paralelas.
Além disso, polígonos ficam mais ``arredondados'' à medida que seu tamanho aumenta,
e a soma de seus ângulos depende da área.
Para triângulos, a soma é maior que $180^\circ$.
Não existem quadrados (quadriláteros com lados iguais e ângulos retos),
e sim quadriláteros equiláteros e com ângulos iguais mas maiores que $90^\circ$ (Fig.\,\ref{fig:quadrado}). 
Sua área não é $\AAA=l^2$, pois se tentarmos dividi-los em outros menores, estes não se encaixarão direito.

\begin{figure}[htb!]
	\centering
	\includegraphics[width=4.5cm]{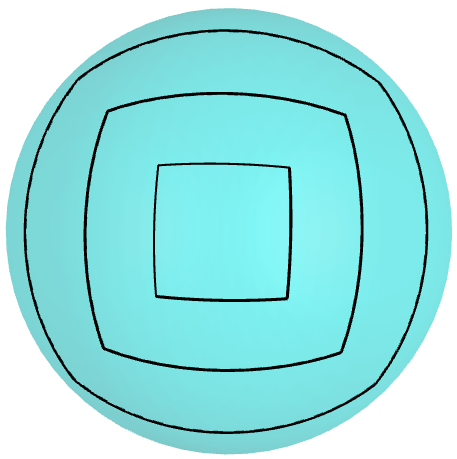}
	\caption{Quadriláteros esféricos equiláteros e equiangulares. 
		Não são semelhantes pois a área afeta os ângulos.}
	\label{fig:quadrado}
\end{figure}


Vamos ver como \eqref{eq:TP} falha na superfície da Terra.
Na Fig.\,\ref{fig:mapa}, temos o triângulo formado por Quito-Equador, Macapá-AP e Porto Alegre-RS.
Duas cidades estão quase na linha do Equador, e duas quase no mesmo meridiano, logo o ângulo em Macapá é quase reto.
Por \eqref{eq:TP}, a distância de Quito a Porto Alegre deveria ser de $4511km$, um erro de $123km$ em relação à distância no mapa.
Isso ocorre porque o triângulo é levemente arredondado, devido à curvatura da Terra, logo o TP não se aplica perfeitamente a ele.


\begin{figure}[htb!]
	\centering
	\includegraphics[width=5.5cm]{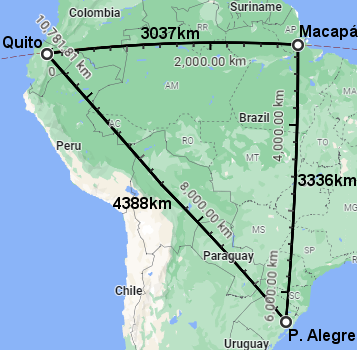}
	\caption{Triângulo retângulo na superfície da Terra. Fonte: Google Maps.}
	\label{fig:mapa}
\end{figure}

O erro aumenta para triângulos maiores.
Na Fig.\,\ref{fig:triretangulo} temos um triângulo esférico delimitando um octante da superfície da Terra.
Ele é equilátero, seus 3 ângulos são retos, e, claro, não há a menor chance do TP funcionar com ele!

\begin{figure}[htb!]
	\centering
	\includegraphics[width=7cm]{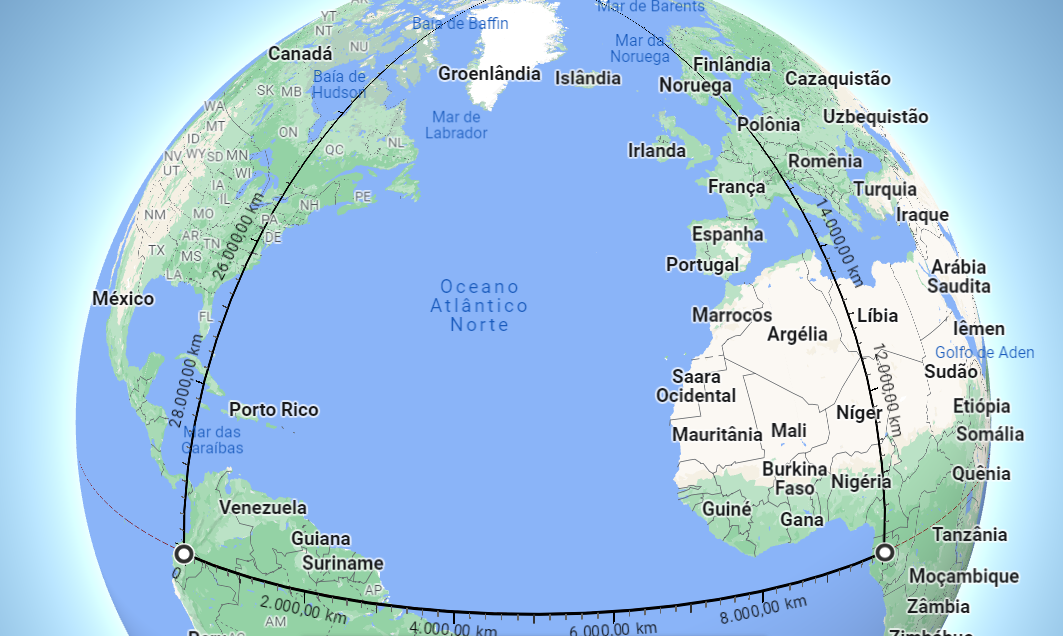}
	\caption{Triângulo esférico equilátero triretângulo. Fonte: Google Maps.}
	\label{fig:triretangulo}
\end{figure}

Alguém poderia dizer que esses não são realmente triângulos, e as distâncias deveriam ser medidas em linha reta, atravessando a Terra.
Mas para o planejamento de rotas, por exemplo, o que importam são distâncias ao longo da superfície.
Para isso a geometria esférica tem seu próprio TP, relacionando os lados de um triângulo retângulo esférico.

\begin{theorem*}[Pitágoras Esférico]
	Sejam $a, b, c$ os comprimentos da hipotenusa e catetos de um triângulo retângulo esférico, 
	numa esfera de raio $R$. 
		\CITE{Wolfe2012 p.198}
	Então
		\SELF{\url{https://math.stackexchange.com/questions/1374058/why-does-the-pythagorean-theorem-have-its-simple-form-only-in-euclidean-geometry} dá versão 3D sem prova: na notação de \eqref{eq:DeGua},
		$\cos \frac{|\AAA|}{2} = \cos \frac{|\AAA_1|}{2} \cos \frac{|\AAA_2|}{2} \cos \frac{|\AAA_3|}{2} + \sin \frac{|\AAA_1|}{2} \sin \frac{|\AAA_2|}{2} \sin \frac{|\AAA_3|}{2}$, e diz que hiperbólico é com $-$ (mas deve ser $\cosh$ e $\sinh$?). Também generaliza lei dos cos.}
	\begin{equation}\label{eq:TP esferico}
		\cos \tfrac aR = \cos \tfrac bR \cdot \cos \tfrac cR.
	\end{equation}
\end{theorem*}
\begin{proof}
	Com a origem $O$ no centro da esfera, e os eixos como na Fig.\,\ref{fig:esferico}, 
	de modo que $A=(R,0,0)$, $B=(R \cos \theta, R \sen \theta,0)$ e $C=(R\cos\beta,0,R\sen\beta)$,
	temos $\alpha = \frac a R$, $\beta = \frac b R$, $\theta = \frac c R$,
	e o produto escalar entre $\overrightarrow{OB}$ e $\overrightarrow{OC}$ leva a
	$\cos \alpha = \cos \beta \cdot \cos \theta$ (exercício).
\end{proof}

\begin{figure}[htb!]
	\centering
	\includegraphics[width=5cm]{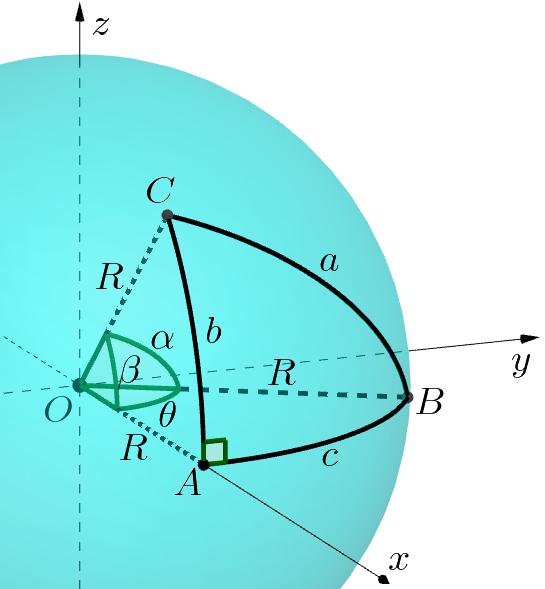}
	\caption{Teor.\ de Pitágoras esférico, $\cos \tfrac aR = \cos \tfrac bR \cdot \cos \tfrac cR$}
	\label{fig:esferico}
\end{figure}

Como o raio da Terra é $R \cong 6371km$,
por \eqref{eq:TP esferico} 
a distância de Porto Alegre a Quito seria de $4414km$, bem mais próxima da real (mas não exata pois o ângulo em Macapá não é precisamente reto, e a Terra não é uma esfera perfeita).

Na Fig.\,\ref{fig:triretangulo} ocorre algo interessante: os 3 lados são hipotenusas! Mas \eqref{eq:TP esferico} funciona, com os três cossenos dando $0$.
Outro fenômeno inusitado é que se $b=\frac\pi2 R$ então $a=\frac\pi2 R$ ou $\frac{3\pi}{2} R$, independente do valor de $c$. Tente entender o significado geométrico disso.

Como o TP sempre serviu para medir pequenas distâncias na superfície da Terra, espera-se que \eqref{eq:TP} e \eqref{eq:TP esferico} deem resultados parecidos para triângulos não muito grandes.
De fato, como $\cos x \cong 1 - \frac{x^2}{2}$ para $x$ pequeno, temos $a^2 \cong b^2 + c^2$ se $a,b,c \ll R$ (exercício).

%

\begin{figure}[htb!]
	\centering
	\includegraphics[width=5cm]{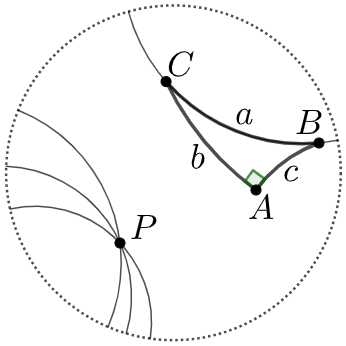}
	\caption{Modelo do disco de Poincaré da geometria hiperbólica. As ``retas'' são arcos de círculo perpendiculares à borda, ou diâmetros do disco. Por $P$ passam infinitas paralelas à reta $BC$. O triângulo $ABC$ é retângulo, com $a\cong 3.34$ e $b=c=2$ (a escala de comprimento aumenta perto da borda). A curvatura é $K=-1$.}
	\label{fig:hyperbolic}
\end{figure}

A geometria hiperbólica é, de certa forma, antípoda da esférica. Nela, por um ponto passam infinitas paralelas a uma reta (Fig.\,\ref{fig:hyperbolic}). 
Triângulos ficam ``magros'', com sua área determinando a soma dos ângulos, que é menor que $180^\circ$.
Sua trigonometria usa funções hiperbólicas, e temos a seguinte versão do TP:

\begin{theorem*}[Pitágoras Hiperbólico]
	Sejam $a, b, c$ os comprimentos da hipotenusa e catetos de um triângulo retângulo hiperbólico, num plano hiperbólico de curvatura $K=-\frac{1}{R^2}$. Então
	\begin{equation}\label{eq:TP hiperbolico}
		\cosh \tfrac aR = \cosh \tfrac bR \cdot \cosh \tfrac cR.
	\end{equation}
\end{theorem*}

A prova pode ser vista em \cite{Wolfe2012}.
	\CITE{p. 144 \\ Agustini2022 p.128}
Teste com os valores da Fig.\,\ref{fig:hyperbolic} ($R$ não é o raio do disco, e sim um pseudo-raio definido a partir de $K$).
Como $\cosh x \cong 1 + \frac{x^2}{2}$ para $x$ pequeno, de novo $a^2 \cong b^2 + c^2$ se $a,b,c \ll R$.

\section{Teorema de Pitágoras Unificado}

Essas aproximações para pequenos triângulos não revelam a conexão mais ampla que há entre os TP's das três geometrias, e que surge ao generalizarmos a forma \eqref{eq:soma areas}.
Mas como, se nas não-euclidianas não há quadrados nem figuras semelhantes em geral?

Semelhança euclidiana requer figuras com comprimentos proporcionais e ângulos iguais,
o que é difícil nas geometrias não-euclidianas, nas quais a área afeta os ângulos.
Mas círculos não têm ângulos, logo esse problema não os atinge, e intuitivamente eles parecem semelhantes.
É uma semelhança mais fraca, pois seus comprimentos não são proporcionais ao raio, nem suas áreas ao quadrado dele (imagine o que ocorre quando o raio do círculo ultrapassa $\frac14$ de uma volta na esfera).
Surpreendentemente, já basta. 

As três geometrias são ligadas pelo conceito de curvatura: na esférica, $K = \frac{1}{R^2}>0$, na euclidiana $K=0$, e na hiperbólica $K = -\frac{1}{R^2}<0$
(a euclidiana pode ser vista como um limite das outras quando $R \rightarrow \infty$).
E uma forma do TP válida em todas é:
\MAYBE{Em \cite{Foote2017} há uma Lei dos Cossenos Unificada. Como relacionar com caso próprio?}

\begin{theorem*}[Pitágoras Unificado \cite{Foote2017}]
	As áreas $\AAA$, $\AAA_1$ e $\AAA_2$ de discos com raios dados pela hipotenusa e catetos de um triângulo retângulo (Fig.\,\ref{fig:proprio 3}) satisfazem
	\begin{equation}\label{eq:TP unificado}
		\AAA = \AAA_1 + \AAA_2 - \frac{K}{2\pi} \AAA_1 \AAA_2.
	\end{equation}
\end{theorem*}
\begin{proof}
	Na geometria esférica, a área do disco de raio $r$ 
	é $\AAA(r) = 2\pi R^2 (1-\cos \frac rR)$, e na hiperbólica é $\AAA(r) = 2\pi R^2 (\cosh \frac rR -1)$.
		\CITE{Wolfe2012 p.169 (hyperbolic)}
	Agora é só usar \eqref{eq:TP esferico} e \eqref{eq:TP hiperbolico}.
\end{proof}

\begin{figure}[htb!]
	\centering
	\includegraphics[width=5cm]{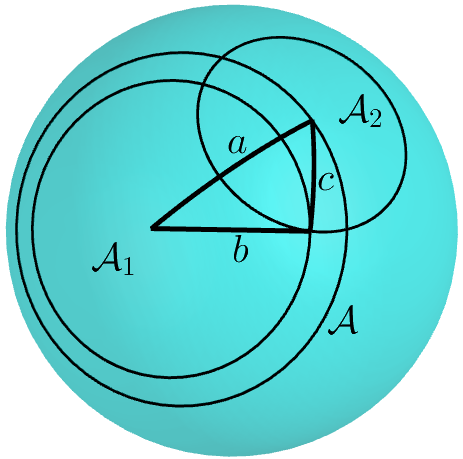}
	\caption{Discos cujos raios são os lados de um triângulo. Se este for retângulo, temos \eqref{eq:TP unificado}. Se for próprio, temos \eqref{eq:soma areas}.}
	\label{fig:proprio 3}
\end{figure}

\section{Triângulos próprios}

O termo extra em \eqref{eq:TP unificado} lembra um pouco o de \eqref{eq:lei cos}, e sugere que triângulos retângulos deixam de ser os que dão a fórmula mais simples.
Na verdade, nas geometrias não-euclidianas eles perdem várias propriedades que os tornam especiais, como a hipotenusa ser o diâmetro do círculo circunscrito.
Mas estas são herdadas por outro tipo de triângulo \cite{Maraner2010}.


Um triângulo é \emph{próprio} se um ângulo for a soma dos demais (Fig.\,\ref{fig:proprio}). O lado oposto a ele também é chamado de \emph{hipotenusa}, e os outros de \emph{catetos}.
Na geometria euclidiana, a soma dos ângulos é $180^\circ$,
logo triângulos retângulos e próprios são a mesma coisa.

\begin{figure}[htb!]
	\centering
	\includegraphics[width=5cm]{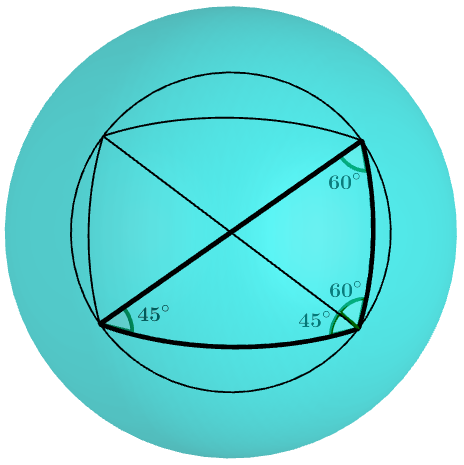}
	\caption{Triângulo próprio, hipotenusa é diâmetro do círculo circunscrito, e diagonal do quadrilátero equiangular}
	\label{fig:proprio}
\end{figure}

Para discos construídos com os lados desse triângulo, \eqref{eq:soma areas} vale nas três geometrias:

\begin{theorem*}[Pitágoras Próprio \cite{Maraner2010}]
	As áreas $\AAA$, $\AAA_1$ e $\AAA_2$ de discos com raios dados pela hipotenusa e catetos de um triângulo próprio satisfazem $\AAA = \AAA_1 + \AAA_2$.
\end{theorem*}

A prova geral é dada em \cite{Maraner2010}.
Em \cite{Pambuccian2010}, são citados artigos anteriores que tratam do caso hiperbólico.
	\CITE{Pambuccian2010 diz que Familiari-Calapso1969 prova caso hiperbólico, e Vranceanu1971 prova $\AAA(a) = \AAA(b) + \frac{\sin (A-B)}{C} \AAA(c)$ para todo trian hiperb. \\
	Maraner2010 diz que como vale nas 3 geometrias, deve dar pra provar com axiomas da Geometria Neutra. Pambuccian2010 diz que precisa também de números reais, e traz refs de Familiari-Calapso que mencionam Geometria Absoluta, mas não achei}

Esse resultado tem uma interpretação interessante: girando o triângulo próprio da Fig.\,\ref{fig:proprio 3} de modo que o segmento $a$ ``varra'' a área $\AAA$, e $b$ varra $\AAA_1$, $c$ irá varrer a área $\AAA-\AAA_1$ do anel entre elas, que é igual à área $\AAA_2$ que $c$ varreria se girasse em torno de sua extremidade. 

O caso euclidiano é fácil de entender \cite{Foote2017}.
Na Fig.\,\ref{fig:dinamico}, ao girar o triângulo retângulo por um ângulo $d\theta$, o segmento $c$ varre uma área $d\AAA$.
Nesse movimento, ele gira o mesmo $d\theta$, enquanto desliza ao longo do seu comprimento. Desconsiderando tal deslizamento, que varre uma área nula, $d\AAA$ será igual à área pintada do disco à direita, obtida só girando $c$.
Integrando, a área do anel será igual à desse disco.

\begin{figure}[htb!]
	\centering
	\includegraphics[width=8.5cm]{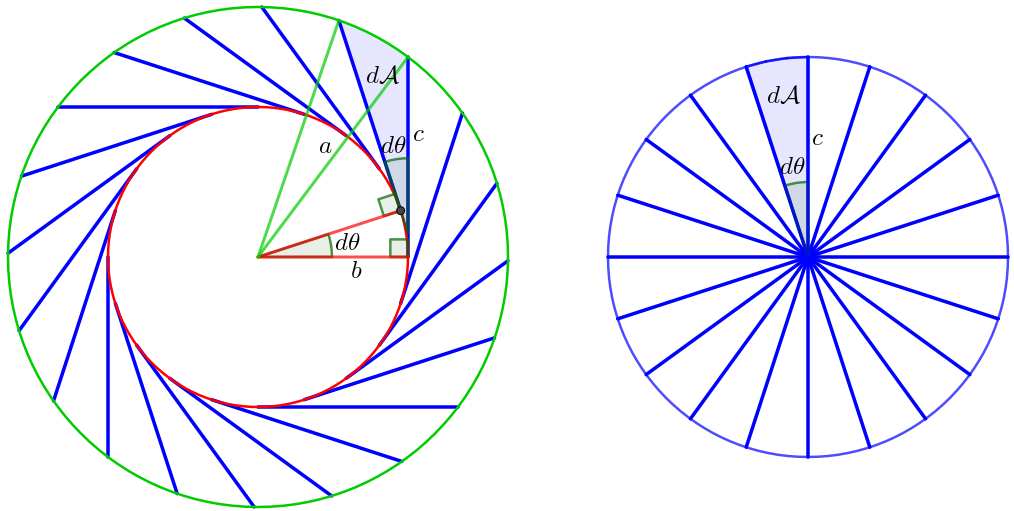}
	\caption{As regiões pintadas têm a mesma área $d\AAA$,
		logo a área total do anel é igual à do disco à direita.}
	\label{fig:dinamico}
\end{figure}

Na Fig.\,\ref{fig:proprio 3}, se o triângulo for retângulo e girar $d\theta$, o segmento $c$ irá girar o mesmo $d\theta$ no espaço, mas sua rotação na superfície da esfera, que varre áreas, será menor (imagine carregar uma vara pela linha do Equador: ela está girando no espaço, mas vista da superfície está sempre apontada para a frente). 
Isso faz a área $\AAA-\AAA_1$ do anel ser menor que $\AAA_2$, o que explica a subtração em \eqref{eq:TP unificado}.
Na geometria hiperbólica, a curvatura $K<0$ faz $c$ girar mais que $d\theta$, de modo que $\AAA-\AAA_1>\AAA_2$.
Nessas duas geometrias, a mudança para triângulo próprio ajusta o ângulo entre $c$ e a direção de rotação para criar um deslizamento transversal ao comprimento, que varre a área necessária para compensar o efeito da curvatura e obter $\AAA-\AAA_1 = \AAA_2$.

Substituindo as fórmulas das áreas, o teorema dá relações entre os lados do triângulo próprio:

\begin{corollary*}
	A hipotenusa $a$ e os catetos $b$ e $c$ de um triângulo próprio esférico ou hiperbólico estão relacionados, respectivamente, por
	\begin{align*}
		1+\cos \tfrac aR &= \cos \tfrac bR + \cos \tfrac cR, \text{ ou} \\
		1+\cosh \tfrac aR &= \cosh \tfrac bR + \cosh \tfrac cR.
	\end{align*}
\end{corollary*}

Teste na Fig.\,\ref{fig:proprio}, construída no GeoGebra com raio da esfera $R=1$, e o triângulo tendo $a=\frac\pi2$, $b\cong 1.22$ e $c\cong 0.86$.
Quando $a,b,c \ll R$, novamente $a^2 \cong b^2 + c^2$.

\section{Tetraedros e simplexos}

Vamos voltar para a geometria euclidiana, e generalizar o TP para dimensões maiores.

Um tetraedro é \emph{trirretângulo} se as três arestas em um vértice são perpendiculares (Fig.\,\ref{fig:tetraedro}). A \emph{face hipotenusal} é a oposta a esse vértice.

\begin{figure}[htb!]
	\centering
	\includegraphics[width=5cm]{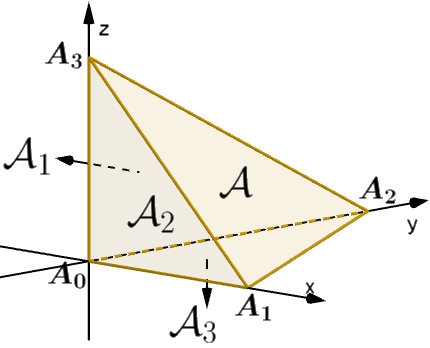}
	\caption{Tetraedro trirretângulo, $\AAA^2 = \AAA_1^2 + \AAA_2^2 + \AAA_3^2$}
	\label{fig:tetraedro}
\end{figure}

\begin{theorem*}[De Gua (1783)]
		\CITE{DE GUA, Jean-Paul. Propositions neuves, et non moins utiles que curieuses, sur le Tétraèdre; ou Essai de Tétraédrométrie In: HISTOIRE de l’Académie Royale des Sciences ... avec lesmémoires de mathématique \& de physique. Paris: Académie Royale des Sciences, 1783. p.363-402. Disponível em: \url{https://gallica.bnf.fr/ark:/12148/bpt6k3582m/f496}}
	Em um tetraedro trirretângulo, a área $\AAA$ da face hipotenusal está relacionada às áreas $\AAA_1$, $\AAA_2$, $\AAA_3$ das outras faces por
	\begin{equation}\label{eq:DeGua}
		\AAA^2 = \AAA_1^2 + \AAA_2^2 + \AAA_3^2.
	\end{equation}
\end{theorem*}
\begin{proof}
	Na Fig.\,\ref{fig:tetraedro}, 
	$\AAA = \frac12 \|\overrightarrow{A_1A_2} \times \overrightarrow{A_1A_3}\|$ e analogamente para as outras áreas.
	Além disso, $\overrightarrow{A_1A_2} = \overrightarrow{A_0A_2} - \overrightarrow{A_0A_1}$
	e
	$\overrightarrow{A_1A_3} = \overrightarrow{A_0A_3} - \overrightarrow{A_0A_1}$.
	O resto fica como exercício (observando que 
	$\overrightarrow{A_0A_1} \times \overrightarrow{A_0A_2}$,
	$\overrightarrow{A_0A_1} \times \overrightarrow{A_0A_3}$ e
	$\overrightarrow{A_0A_2} \times \overrightarrow{A_0A_3}$
	são perpendiculares).
\end{proof}


A versão em $\R^n$ desse tetraedro é um \emph{$n$-simplexo retangular}.
Em um sistema de coordenadas cartesianas ortogonais $(x_1,\ldots,x_n)$, ele é um poliedro $n$-dimensional com um vértice $A_0$ na origem, e um $A_k$ em cada eixo, na coordenada $x_k = a_k>0$.
Sua  \emph{face} $F_k$ é o poliedro $(n-1)$-dimensional formado por todos os vértices exceto $A_k$.
A \emph{face hipotenusal} é $F_0$.

O seguinte resultado é frequentemente redescoberto, por diferentes meios
(\cite{Cho1991,Donchian1935,Eifler2008,Yeng1990}, entre outros).
Um $m$-volume é um volume $m$-dimensional (comprimentos são 1-volumes, áreas são 2-volumes, etc.).

\begin{theorem*}
	Em um $n$-simplexo retangular $S$, seja $\VV_k$ o $(n-1)$-volume da face $F_k$. Então 
		\SELF{Cho1991 relaciona $m$-volumes de sub-simplexos de $S$ (por ex., na Fig.\,\ref{fig:tetraedro}, a soma dos quadrados das arestas da face $\AAA$ é o dobro da soma dos quadrados das outras arestas).}
	\begin{equation}\label{eq:simplexo}
		\VV_0^2 = \sum_{k=1}^n \VV_k^2.
	\end{equation}
\end{theorem*}

Daremos uma prova via Cálculo Vetorial, e outra usando Cálculo Integral e Geometria Analítica.

\begin{proof}[Demonstração (\cite{Eifler2008}).]
	Seja $\vec{v}$ um campo vetorial constante em $\R^n$,
	$dV$ o diferencial de $n$-volume, $d\VV$ o de $(n-1)$-volume, 
	e $\hat{n}_k$ o vetor normal unitário para fora de $F_k$. 
	Como $\Div \vec{v} = 0$, o Teorema da Divergência dá
	\begin{align*}
		0 &= \int_S \Div \vec{v} \,dV 
		= \sum_{k=0}^n \int_{F_k} \vec{v}\cdot\hat{n}_k \,d\VV 
		= \sum_{k=0}^n \vec{v}\cdot\hat{n}_k \int_{F_k}  \,d\VV \\
		&= \sum_{k=0}^n \vec{v}\cdot\hat{n}_k \VV_k 
		= \vec{v}\cdot \left(\sum_{k=0}^n \VV_k \hat{n}_k\right)  
	\end{align*}
	Como $\vec{v}$ é arbitrário, isso implica $\sum_{k=0}^n \VV_k \hat{n}_k = 0$, logo 
	$\VV_0 \hat{n}_0 = - \VV_1 \hat{n}_1 - \cdots - \VV_n \hat{n}_n$. Como $\hat{n}_1, \ldots, \hat{n}_n$ são ortogonais,
	\eqref{eq:norma} nos dá o resultado.
\end{proof}

\begin{proof}[Demonstração (\cite{Donchian1935}).]
	Para $0\leq k\leq n$, seja $h_k$ a altura do vértice $A_k$ em relação à face oposta $F_k$.
	A seção transversal a essa altura, a uma distância $x$ de $A_k$, é semelhante a $F_k$, e tem $(n-1)$-volume $\VV(x) = \VV_k \cdot \left(\frac{x}{h_k}\right)^{n-1}$. 
	Assim, o $n$-volume de $S$ é $V = \int_0^{h_k} \VV(x) dx =\frac{\VV_k h_k}{n}$, para todo $k$.
		\SELF{Ou $V = \frac{a_1\cdots a_n}{(n+1)!}$, que vem de decompor o paralelotopo gerado a partir das $n$ arestas $A_0 A_k$ em um total de $(n+1)!$ simplexos de volume $V$}
	Para $k\neq 0$ temos $h_k = a_k$ (o comprimento da aresta $A_0A_k$),
	logo $\VV_k a_k = n V = \VV_0 h_0$, ou seja, $\VV_k = \VV_0 \cdot \frac{h_0}{a_k}$.
	
	O hiperplano da face $F_0$ é descrito pela equação $\frac{x_1}{a_1} + \cdots + \frac{x_n}{a_n} = 1$.
	Como $\vec{v} = (\frac{1}{a_1},\ldots,\frac{1}{a_n})$ é ortogonal a ele,
	e $\lambda\vec{v}$ satisfaz a equação quando $\lambda=\left(\frac{1}{a_1^2} + \cdots + \frac{1}{a_n^2}\right)^{-1}$, temos 
	$h_0 = \|\lambda\vec{v}\| = \left(\frac{1}{a_1^2} + \cdots + \frac{1}{a_n^2}\right)^{-\frac12}$.
	Assim, obtemos
	$\sum_{k=1}^n \VV_k^2 
	= \VV_0^2 h_0^2 \cdot \sum_{k=1}^n  \frac{1}{a_k^2}
	= \VV_0^2$.
\end{proof}
	

Outros métodos serão vistos a seguir, quando olharemos para o TP e estes teoremas de outra forma.

\section{Projeções de volumes}

Esses resultados podem ser reinterpretados em termos de projeções ortogonais de $m$-volumes.
Sempre que falarmos em projeções, serão ortogonais.

Separando os lados do triângulo retângulo, podemos ver o TP como uma relação entre o comprimento de um segmento reto e suas projeções em um sistema de eixos perpendiculares, como na Fig.\,\ref{fig:proj linha eixos}. 

\begin{figure}[htb!]
	\centering
	\includegraphics[width=8cm]{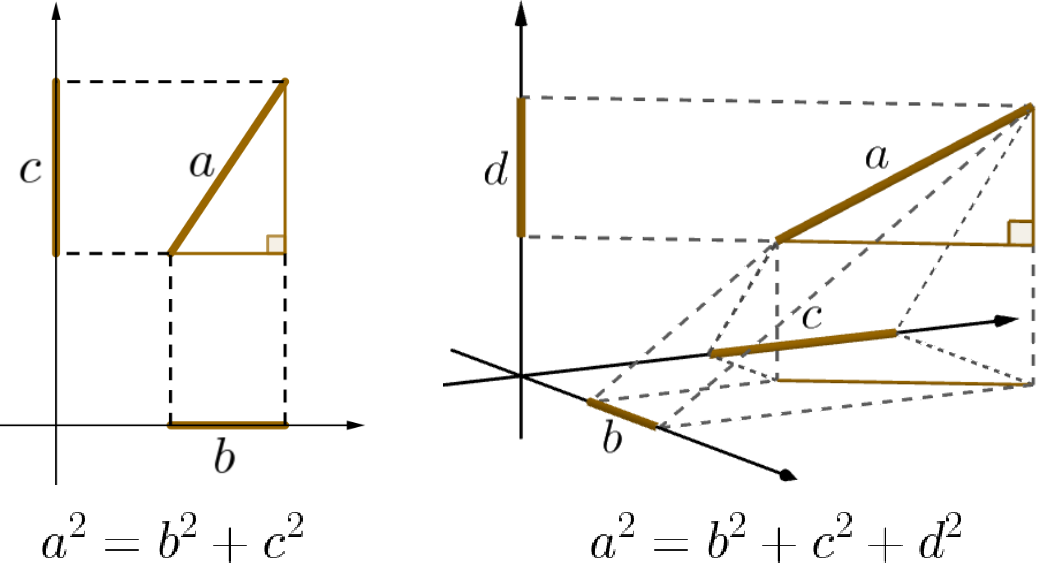}
	\caption{Projeções ortogonais de segmentos sobre eixos}
	\label{fig:proj linha eixos}
\end{figure}

E, separando as faces do tetraedro, \eqref{eq:DeGua} vira uma relação entre a área de um triângulo e suas projeções em planos coordenados perpendiculares.
Como projeções são transformações lineares,
a razão $\frac{\AAA_k}{\AAA}$ não muda se trocarmos $\AAA$ pela área de qualquer figura no mesmo plano, e $\AAA_k$ pela de sua projeção no $k$-ésimo plano coordenado.
Assim, \eqref{eq:DeGua} vale para qualquer área plana e suas projeções (Fig.\,\ref{fig:projecao planos}).
Esse resultado já era conhecido no Séc. XVIII.
	\CITE{p.450 Eves1990}

\begin{figure}[htb!]
	\centering
	\includegraphics[width=5cm]{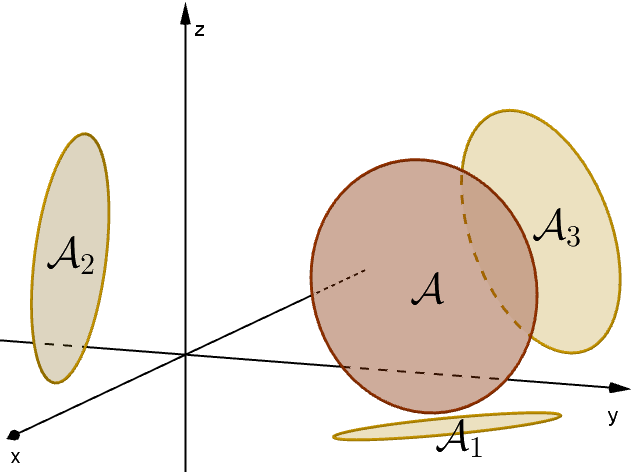}
	\caption{Projeções de uma área plana, $\AAA^2 = \AAA_1^2 + \AAA_2^2 + \AAA_3^2$}
	\label{fig:projecao planos}
\end{figure}

Da mesma forma, \eqref{eq:simplexo} vale para qualquer $(n-1)$-volume $\VV_0$ num hiperplano de $\R^n$, e suas projeções $\VV_k$ em hiperplanos coordenados perpendiculares.

Vamos estender para $m$-volumes.
Sejam $(e_1,\ldots,e_n)$ uma base ortonormal, 
e $\II_m$ o conjunto dos \emph{multi-índices} $I=(i_1,\ldots,i_m)$ com $1\leq i_1<i_2<\cdots<i_m \leq n$.
Cada $C_I = \Span\{e_i:i\in I\}$, com $I \in \II_m$, é um \emph{subespaço coordenado}.
Chamaremos de \emph{figura} ou \emph{região} os conjuntos mensuráveis, i.e.\ com um volume bem definido da dimensão do subespaço considerado.

\begin{theorem*}[\cite{Conant1974}]
	Se $\VV$ é o $m$-volume de uma região $R$ num subespaço $m$-dimensional de $\R^n$,
	e $\VV_I$ é o de sua projeção em $C_I$, então
	\begin{equation}\label{eq:Conant Beyer}
		\VV^2 = \sum_{I \in \II_m} \VV_I^2.
	\end{equation}
\end{theorem*}

Pela linearidade das projeções, basta provar para um paralelotopo (paralelogramo $m$-dimensional) $R$ gerado por vetores $v_1,\ldots,v_m$.
Vamos ver primeiro uma prova via determinantes.

\begin{proof}[Demonstração (\cite{Conant1974}).]
	Pondo $v_1,\ldots,v_m$ como colunas de uma matriz $M_{n \times m}$, o teorema de Cauchy-Binet 
		\CITE{Gantmacher2000 p.9} dá
	\begin{equation}\label{eq:Cauchy-Binet}
		\det(M^T M) = \sum_{I \in \II_m} \det(M_I)^2,
	\end{equation}
	onde $M_I$ é a submatriz $m \times m$ de $M$ formada pelas linhas com índices em $I$, de modo que suas colunas são as projeções dos $v_k$'s em $C_I$.
	O resultado vem da interpretação geométrica dos determinantes:
	$\det(M^T M)$ é o determinante de Gram que dá o quadrado do $m$-volume de $R$, e $|\det(M_I)|$ dá o $m$-volume da projeção de $R$ em $C_I$.
\end{proof}

A prova mais fácil usa Álgebra Exterior \cite{Khosravi2008}, uma extensão da Álgebra Linear com \emph{multivetores}, vetores representando figuras multi-dimensionais (Fig.\,\ref{fig:multivectors}).

\begin{figure}[htb!]
	\centering
	\includegraphics[width=7cm]{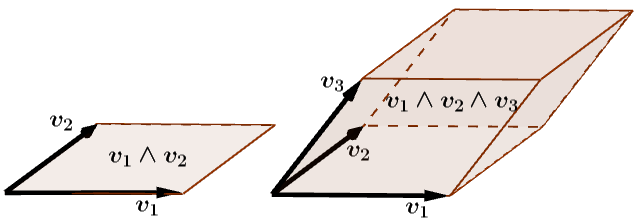}
	\caption{Multivetores: $v_1\wedge v_2$ representa um paralelogramo de área $\|v_1\wedge v_2\|$, e $v_1\wedge v_2 \wedge v_3$ representa um paralelepípedo de volume $\|v_1\wedge v_2 \wedge v_3\|$.}
	\label{fig:multivectors}
\end{figure}

\begin{proof}
	O multivetor $M = v_1\wedge \cdots \wedge v_m$ representa $R$, e a norma $\|M\|$ é seu $m$-volume $\VV$.
	Multivetores são também vetores num certo espaço vetorial, 
	que tem base ortonormal formada pelos $e_{i_1} \wedge \cdots \wedge e_{i_m}$ com $I=(i_1,\ldots,i_m) \in \II_m$.
	Cada componente de $M$ nessa base representa a projeção de $R$ em $C_I$, de modo que \eqref{eq:Conant Beyer} vem de \eqref{eq:norma}.
\end{proof}

\begin{example}
	Se $\AAA$ é a área de uma figura plana em $\R^4=\{(x_1,\ldots,x_4)\}$, e  $\AAA_{ij}$ é sua projeção no plano $x_i x_j$, 
	\begin{equation}\label{eq:areas R4}
		\AAA^2 = \AAA^2_{12} + \AAA^2_{13} + \AAA^2_{14} + \AAA^2_{23} + \AAA^2_{24} + \AAA^2_{34}.
	\end{equation}
	Por exemplo, $v=(a,b,c,d)$ e $w=(-b,a,-d,c)$ têm mesma norma e são perpendiculares, formando um quadrado de área $\AAA = a^2+b^2+c^2+d^2$. 
	Suas projeções $(a,b,0,0)$ e $(-b,a,0,0)$ no plano $x_1 x_2$ formam um quadrado de área $\AAA_{12} = a^2+b^2$. 
	Da mesma forma, $\AAA_{34} = c^2 + d^2$.
	As projeções $(a,0,c,0)$ e $(-b,0,-d,0)$ no plano $x_1 x_3$ formam um paralelogramo com $\AAA_{13} = |bc-ad|$ (ignore a 4\textordfeminine\ coordenada nula, e use o produto vetorial).
	Analogamente, obtemos $\AAA_{24} = \AAA_{13}$ e $\AAA_{14} = \AAA_{23} = |ac+bd|$.
	Uma conta chata mostra que essas áreas satisfazem \eqref{eq:areas R4}. 
\end{example}

\begin{example}
	Se $\VV$ é o volume de uma região em um subespaço tridimensional de $\R^4$, e  $\VV_{ijk}$ é sua projeção no subespaço $x_i x_j x_k$, 
	\begin{equation*}
		\VV^2 = \VV^2_{123} + \VV^2_{124} + \VV^2_{134} + \VV^2_{234}.
	\end{equation*}
\end{example}

Com um pouco de combinatória, generalizamos para projeções em subespaços de outra dimensão.

\begin{corollary*}[\cite{Czyzewska1991,Drucker2015}]
	Se $\VV$ é o $p$-volume de uma região $R$ em um subespaço de dimensão $p\leq m$ de $\R^n$,
	e $\VV_I$ é sua projeção em $C_I$, com $I \in \II_m$, então 
	\begin{equation*}
		\VV^2 = \textstyle{\binom{n-p}{n-m}^{-1}} \cdot \displaystyle{\sum_{I \in \II_m} \VV_I^2}.
	\end{equation*}
\end{corollary*}
\begin{proof}
	Cada $\VV_I$ pode ser projetado nos subespaços coordenados de dimensão $p$ contidos em $C_I$,
	e cada um desses está em $\binom{n-p}{n-m}$ dos $C_I$'s.
	O resultado então segue de \eqref{eq:Conant Beyer}.
\end{proof}

\begin{example}
	Sejam $\LL$ o comprimento de um segmento em $\R^3$, 
	$\LL_1$, $\LL_2$ e $\LL_3$ suas projeções nos planos coordenados,
	e $b$, $c$ e $d$ suas projeções nos eixos  (Fig.\,\ref{fig:line planes}).
	Então $\LL_1^2 = b^2+c^2$, $\LL_2^2 = b^2+d^2$, $\LL_3^2 = c^2+d^2$ e $\LL^2 = b^2+c^2+d^2 = \frac12 (\LL_1^2 + \LL_2^2 + \LL_3^2)$.
\end{example}

\begin{figure}[htb!]
	\centering
	\includegraphics[width=6cm]{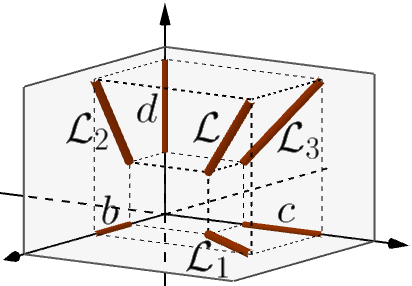}
	\caption{Projeções de um segmento nos planos coordenados,
		$\LL^2 = \frac12 (\LL_1^2 + \LL_2^2 + \LL_3^2)$.}
	\label{fig:line planes}
\end{figure}

\section{Versões complexas}

Em espaços complexos, a generalização \eqref{eq:norma} do TP pode ser reinterpretada como uma soma de áreas, e estendida para volumes de maior dimensão, com fórmulas mais simples que as do caso real.

Vamos rever a geometria real desses espaços.
Pode-se identificar $z = x + \im y \in \C$ com o vetor $(x,y) \in \R^2$, e seu módulo é a norma do vetor, ou seja, $|z|^2 = x^2 + y^2$.
O que diferencia $\C$ de $\R^2$ é a multiplicação por $\im$, que pode ser vista apenas como uma rotação de $90^\circ$, pois $\im z = -y + \im x$ corresponde ao vetor $(-y,x)$, de mesma norma e ortogonal a $(x,y)$. 
Assim, $\im^2=-1$ significa apenas que $\im^2(x,y) = \im(-y,x)=(-x,-y)$.  

Da mesma forma, identifica-se $v = (z_1,\ldots,z_n) \in \C^n$, onde $z_k = x_k + \im y_k$, com $(x_1,y_1,\ldots,x_n,y_n) \in \R^{2n}$, e 
$\|v\|^2=|z_1|^2+\cdots+|z_n|^2 = x_1^2+y_1^2+\cdots+x_n^2+y_n^2$.
De novo, $\im v = (\im z_1,\ldots,\im z_n)$ corresponde a $(-y_1,x_1,\ldots,-y_n,x_n)$, que é $v$ girado $90^\circ$ numa certa direção.
Assim, $\C^n$ é só $\R^{2n}$ com um operador $\im$ de rotação de $90^\circ$. 

Se $v \in \C^n$ e $z = x+ \im y \in \C$ então $zv = x v + y(\im v)$ é combinação linear (com coeficientes reais $x$ e $y$) de $v$ e $\im v$.
Assim, a \emph{linha}%
\footnote{Por  ter dimensão complexa 1.}
\emph{complexa} $\C v = \{zv:z\in \C\}$ é o plano real $\Span\{v,\im v\}$,
no qual esses vetores formam um quadrado de área $\AAA = \|v\|^2$ (Fig.\,\ref{fig:complex square}).

\begin{figure}[htb!]
	\centering
	\includegraphics[width=6cm]{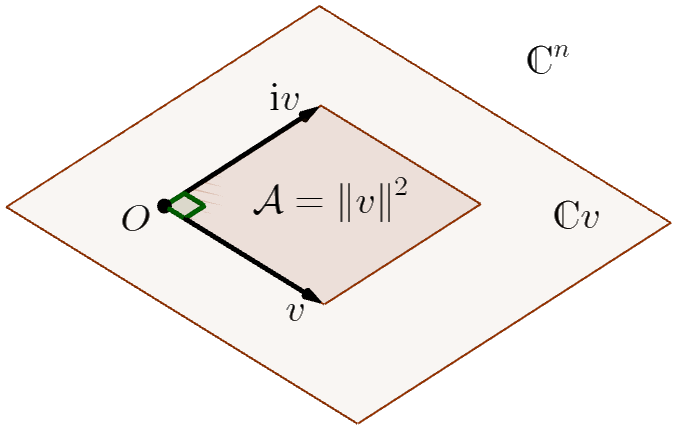}
	\caption{Quadrado numa linha complexa}
	\label{fig:complex square}
\end{figure}

Seja $(e_1,\ldots,e_n)$ uma base ortogonal de $\C^n$.

\begin{theorem*}[Pitágoras para Linhas Complexas \cite{Mandolesi_Pythagorean}]
	Se $\AAA$ é a área de uma região em uma linha complexa $\C v$, e $\AAA_k$ é sua projeção em $\C e_k$, então
	\begin{equation}\label{eq:soma areas complexo}
		\AAA = \AAA_1 + \cdots + \AAA_n.
	\end{equation}
\end{theorem*}
\begin{proof}
	Basta provar para o quadrado da Fig.\,\ref{fig:complex square}.
	Como a projeção $P:\C v \rightarrow \C e_k$ é $\C$-linear, se $P(v) = v_k$ então $P(\im v) = \im v_k$.
	Logo, a projeção de $\AAA$ é outro quadrado, de área $\AAA_k = \|v_k\|^2$.
	O resultado segue de \eqref{eq:norma}.
\end{proof}

Ao contrário do caso real, \eqref{eq:soma areas complexo} não é quadrática, e tem menos termos, já que a dimensão complexa é metade da real correspondente.
Na Fig.\,\ref{fig:complex lines}, 
$\AAA$ é a soma de 2 áreas, das projeções nas linhas complexas de uma base ortogonal $(e_1,e_2)$ de $\C^2$. 
Identificando $\C^2 \cong \R^4$, teríamos \eqref{eq:areas R4}, com 6 termos quadráticos.

\begin{figure}[htb!]
	\centering
	\includegraphics[width=5cm]{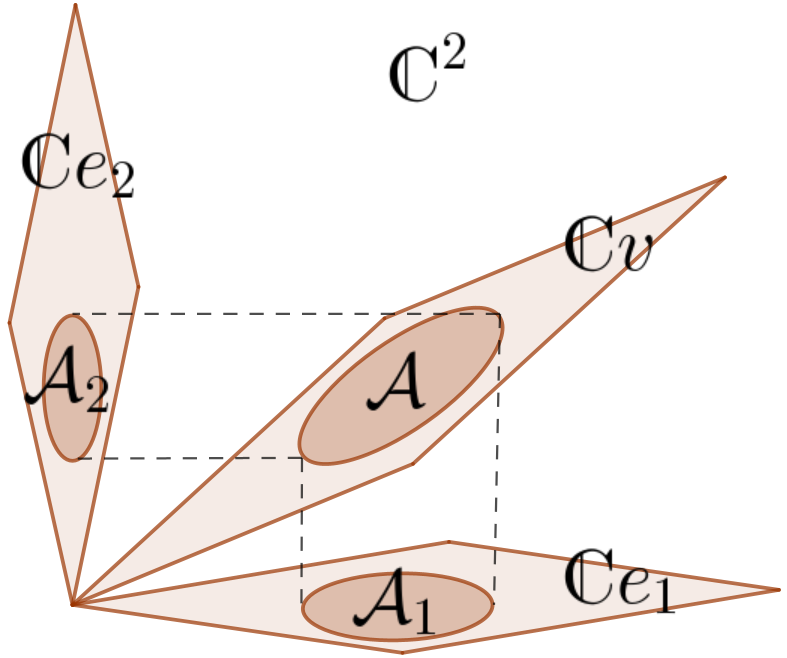}
	\caption{Projeção em linhas complexas, $\AAA = \AAA_1 + \AAA_2$}
	\label{fig:complex lines}
\end{figure}

\begin{example}
	Em $\C^2$, os vetores $v=(a+ib,c+id)$ e $\im v =(-b+\im a,-d+\im c)$, com $a,b,c,d\in\R$, têm norma $\sqrt{a^2+b^2+c^2+d^2}$,
	e formam um quadrado de área $\AAA = a^2+b^2+c^2+d^2$.
	Sejam $e_1=(1,0)$ e $e_2=(0,1)$.
	As projeções de $v$ e $ \im v$ em $\C e_1$ são $(a+\im b,0)$ e $(-b+\im a,0)$, e formam um quadrado de área $\AAA_1 = a^2+b^2$.
	E em $\C e_2$ formam um quadrado de área $\AAA_2 = c^2 + d^2$.
	Como previsto, $\AAA = \AAA_1 + \AAA_2$.

	Identificando $\C^2 \cong \R^4$, este se torna o exemplo dado após \eqref{eq:areas R4}, com 	
	$w = iv$. 
	Note como a abordagem real foi mais complexa%
	\footnote{Juro que tentei resistir.}.
	Naquele exemplo, ao planos $x_1 x_2$ e $x_3 x_4$ correspondem a $\C e_1$ e $\C e_2$, por isso as projeções deram quadrados.
	O plano $x_1 x_3$ não corresponde a uma linha complexa, pois $\im (1,0,0,0) = (0,1,0,0)$ sai dele, e o mesmo ocorre com os outros nos quais as projeções deram paralelogramos.
\end{example}

Vamos estender para um subespaço complexo $V$ de dimensão $m$,
que é um subespaço real de dimensão $2m$ invariante pela rotação $\im$, isto é, $v \in V \Rightarrow \im v \in V$.
Como antes, seja $C_I = \Span\{e_i:i\in I\}$ para $I \in \II_m$ (mas agora o $\Span$ é com coeficientes complexos).

\begin{theorem*}[Pitágoras para Subespaços Complexos \cite{Mandolesi_Pythagorean}]
	Se $\VV$ é o $2m$-volume de uma região $R$ num subespaço complexo $m$-dimensional em $\C^n$, e $\VV_I$ é sua projeção em $C_I$, então
	\begin{equation}\label{eq:complex volumes}
		\VV = \sum_{I \in \II_m} \VV_I.
	\end{equation}
\end{theorem*}
\begin{proof}
	Como a de \eqref{eq:Conant Beyer}, mas \eqref{eq:Cauchy-Binet} se torna
		$\det(M^\dagger M) = \sum_{I \in \II_m} |\det(M_I)|^2$,
	onde $M^\dagger$ é a transposta conjugada,
	e a interpretação de determinantes complexos muda:
	$\det(M^\dagger M)$ é o $2m$-volume do paralelotopo formado por $v_1, \im v_1, \ldots,v_m, \im v_m$,
	e sua projeção em $C_I$ é $|\det(M_I)|^2$.
	
	Também pode-se usar \eqref{eq:norma} e Álgebra Exterior complexa, na qual $\|v_1\wedge \cdots \wedge v_m\|^2$ é o $2m$-volume do paralelotopo formado por $v_1, \im v_1, \ldots,v_m, \im v_m$.
\end{proof}


\section{Conclusão}

Há mais generalizações que não discutimos, como versões da lei dos cossenos para áreas ou volumes \cite{Cho1991,Khosravi2008}, ou em geometrias não-euclidianas \cite{Foote2017}.
Outras esperam para serem descobertas:
por ex., nos parece que \eqref{eq:Conant Beyer} e \eqref{eq:complex volumes} virem $\sqrt{\VV} = \sum_{I \in \II_m} \sqrt{\VV_I}$ em espaços quaterniônicos%
\footnote{Quatérnios generalizam os complexos, com 4 números reais.};
e há menções (sem prova) de uma versão do Teorema de De Gua para geometrias não-euclidianas de dimensão 3, mas não para dimensões maiores.
	\CITE{\url{https://math.stackexchange.com/questions/1374058/why-does-the-pythagorean-theorem-have-its-simple-form-only-in-euclidean-geometry}}
Também seria interessante provar o Teorema de Pitágoras Próprio a partir dos axiomas da Geometria Absoluta (a grosso modo, a parte comum das geometrias euclidiana, elíptica e hiperbólica).

É impressionante que, milênios após suas primeiras aparições, o TP ainda guarde surpresas e mistérios.
Isso mostra que esse velho teorema, com suas mil faces, continua a evoluir e dar frutos.
Vida longa ao Teorema de Pitágoras!



\begin{thebibliography}{10}
	
	\bibitem{Bogomolny}
	A.~Bogomolny.
	\newblock {P}ythagorean {T}heorem.
	\newblock \url{https://www.cut-the-knot.org/pythagoras}.
	\newblock Accessed: 2023-10-04.
	
	\bibitem{Cho1991}
	E.~C. Cho.
	\newblock The generalized cross product and the volume of a simplex.
	\newblock {\em Appl. Math. Lett.}, 4(6):51--53, 1991.
	
	\bibitem{Conant1974}
	D.~R. Conant and W.~A. Beyer.
	\newblock Generalized {P}ythagorean theorem.
	\newblock {\em Amer. Math. Monthly}, 81(3):262--265, 1974.
	
	\bibitem{Czyzewska1991}
	K.~Czyzewska.
	\newblock Generalization of the {P}ythagorean theorem.
	\newblock {\em Demonstratio Math.}, 24(1–2), 1991.
	
	\bibitem{Donchian1935}
	P.~S. Donchian and H.~S.~M. Coxeter.
	\newblock An n-dimensional extension of {P}ythagoras' theorem.
	\newblock {\em Math. Gaz.}, 19(234):206--206, 1935.
	
	\bibitem{Drucker2015}
	D.~Drucker.
	\newblock A comprehensive {P}ythagorean theorem for all dimensions.
	\newblock {\em Amer. Math. Monthly}, 122(2):164--168, 2015.
	
	\bibitem{Eifler2008}
	L.~Eifler and N.~H. Rhee.
	\newblock The n-dimensional {P}ythagorean theorem via the divergence theorem.
	\newblock {\em Amer. Math. Monthly}, 115(5):456--457, 2008.
	
	\bibitem{Foote2017}
	R.~L. Foote.
	\newblock A unified {P}ythagorean theorem in {E}uclidean, spherical, and
	hyperbolic geometries.
	\newblock {\em Math. Mag.}, 90(1):59--69, feb 2017.
	
	\bibitem{Khosravi2008}
	M.~Khosravi and M.~D. Taylor.
	\newblock The wedge product and analytic geometry.
	\newblock {\em Amer. Math. Monthly}, 115(7):623--644, 2008.
	
	\bibitem{Yeng1990}
	S.~Y. Lin and Y.~F. Lin.
	\newblock The n-dimensional {P}ythagorean theorem.
	\newblock {\em Linear Multilinear Algebra}, 26(1-2):9--13, 1990.
	
	\bibitem{Loomis1968}
	E.~S. Loomis.
	\newblock {\em The {P}ythagorean Proposition}.
	\newblock National Council of Teachers of Mathematics, Inc., 1968.
	
	\bibitem{Mandolesi_Pythagorean}
	A.~L.~G. Mandolesi.
	\newblock Projection factors and generalized real and complex
	{P}y\-thag\-o\-re\-an theorems.
	\newblock {\em Adv. Appl. Clifford Algebras}, 30(43), 2020.
	
	\bibitem{Maraner2010}
	P.~Maraner.
	\newblock A spherical {P}ythagorean theorem.
	\newblock {\em Math. Intelligencer}, 32(3):46--50, 09 2010.
	
	\bibitem{Pambuccian2010}
	V.~Pambuccian.
	\newblock {M}aria {T}eresa {C}alapso’s hyperbolic {P}ythagorean theorem.
	\newblock {\em Math. Intelligencer}, 32(4):2--2, 2010.
	
	\bibitem{Roque2012}
	Tatiana Roque.
	\newblock {\em Hist{\'o}ria da matem{\'a}tica: uma vis{\~a}o cr{\'i}tica,
		desfazendo mitos e lendas}.
	\newblock Editora Zahar, 2012.
	
	\bibitem{Wolfe2012}
	H.~E. Wolfe.
	\newblock {\em Introduction to non-Euclidean geometry}.
	\newblock Courier Corporation, 2012.
	
\end{thebibliography}

\end{document}